\documentclass{article}
\usepackage[nottoc,numbib]{tocbibind}
\usepackage[english]{babel}
\usepackage[letterpaper,top=2cm,bottom=2cm,left=3cm,right=3cm,marginparwidth=1.75cm]{geometry}
\usepackage{amsmath}
\usepackage{amsthm}
\usepackage[colorlinks=true, allcolors=black]{hyperref}
\usepackage{graphicx} 
\usepackage{comment}
\usepackage{mathabx}
\usepackage{listings}
\usepackage{xcolor}
\usepackage{tikz}
\usepackage{tikzit}
\usepackage{titlesec}
\usepackage{url}

\tikzstyle{black dot}=[fill=black, draw=black, shape=circle, minimum size = 0.2cm, inner sep = 0pt]

\newtheorem{theorem}{Theorem}
\theoremstyle{definition}

\title{The Game of Cycles with Sources Allowed}
\author{VIGYAN SAHAI \and RAVI TRIPATHI}
\date{August 2023}

\begin{document}
\maketitle

\begin{abstract}
        In this paper we introduce a variant of Francis Su's ``Game of Cycles,'' that we call ``Cycles with Sources.'' The only change to the rules is permitting nodes to be sources, while sinks are still prohibited. Despite this minor change in the rules, we show that even on simple games, like line graphs, there is a great change in the outcome of optimal play, which we fully analyze using Sprague–Grundy Theory.
\end{abstract}

\tableofcontents

\section{Introduction}
    
   \indent Francis Su's "Game of Cycles"\cite{su2020mathematics} has the property that for many classes of simple graphs, the  ``parity conjecture''~\cite{alvarado2021game} is true: that is, for a graph with an even number of markable edges, the second player can always win, and for a graph with an odd number of markable edges, the first player can always win.
    \newline
    \indent In this work, we ask: how sensitive is this phenomenon to small changes in the rules of the game? In particular, in the game of cycles, any move that creates a source node or a sink node is illegal. We ask: what if only moves that create sink nodes are illegal, but moves that create source nodes are allowed? We call this variant of Cycles, ``Cycles with Sources''. Changing this removes the symmetry that was at play with the original cycles game, which could often be won by ``mirroring'' your opponent's moves. By removing this symmetry, we throw the parity conjecture into question. 
    \newline
    \indent  We analyze games on simple graphs like lines using the Sprague–Grundy Theorem, which states that any impartial computational game is equivalent to a game of Nim of size $n$. We deal with this $n$, called the Grundy number or Nimber. Player 1 loses if the Nimber is 0, and any other Nimber greater than 0 results in Player 1's victory. This is because any game with a Nimber greater than 0 can be reduced into a game with Nimber 0. The second player now becomes the ``first player'' and loses. The process of finding Nimbers is explained further in detail in section 2, however, we recommend being familiar with Sprague-Grundy Theory.
    \newline
    \indent Ultimately for line graphs, this simple rule change turned a game that had Nimber 1 for odd sizes and 0 for even sizes into a game whose Nimbers followed a repeated cycle of length 17, starting at games of size 19, and containing Nimbers as high as 8, as calculated by a computer program shown in section 5. In proving that the observed repetition continues forever, we established a more general principle for any game whose moves always divide it into sub-games in a certain way, where these sub-games are also similarly sub-dividable, showing that repeating sequences can be shown to repeat indefinitely, provided the repetition occurs for long enough.
    \newline
    \indent Cycles with Sources can also simplify the outcome of certain games. In section 4, we analyze a game consisting of a single cycle, which with the regular rule-set follows the parity conjecture. However, in Cycles with Sources, the rule change turned the game into a guaranteed victory for player no matter the size of $n$. However, unlike with sources disallowed, the winning strategy is not as simple, as Nimbers on the second move can reach as high as 9.
    \newline
    \indent In Section 2 we analyze the basic game of cycles on a line, which is already understood, but it informs our approach in the next section, in which we apply the new rule-set to a line and examine how the Nimbers of the same types of sub-games are affected. Theorem 2, proved in Section 3, then informs other simple cases with the sources-allowed rule-set, demonstrated in Section 4. Further questions about the game of cycles can be found in~\cite{lin2021exploring,fokkink2022some}.

\section{Simple Examples of the Game of Cycles}
We start with a case that has already been analyzed, specifically by Mathews~\cite{mathews2021game}, the standard game of cycles with its normal rule-set on a straight line of $n$ edges (this game was also analyzed by Lin~\cite{lin2021exploring} with respect to whether the first or second player wins, however, without Sprague-Grundy Theory), simply to gain an understanding of the way Nimbers relate to the Nimbers of sub-games of various ``types.''

\begin{theorem}
    \cite{mathews2021game} A line segment of length $n$ has Nimber 0 if $n$ is even, and Nimber 1 if $n$ is odd.
\end{theorem}

\begin{proof}
    We will prove the theorem by induction. We refer to a game of type $i$ with $n$ unmarked segments as $g_i(n)$ from now on. All games other than $g_1(n)$ are sub-games where the directed segment at the end is not restricted by the no source or sink rule. Here is our induction hypothesis for all $n>1$:

        \begin{enumerate}
            \item $\bullet$---$\bullet\dots\bullet$---$\bullet$    : $n$ unmarked segments: Nimber is $0$ for $n$ even and $1$ for $n$ odd.
        
            \item $\bullet\rightarrow\bullet$---$\bullet\dots\bullet$---$\bullet$  : $n$ unmarked segments: Nimber is $n-1$.

            \item $\bullet\leftarrow\bullet$---$\bullet\dots\bullet$---$\bullet$  : $n$ unmarked segments: Nimber is $n-1$.

            \item $\bullet\rightarrow\bullet$---$\bullet\dots\bullet$---$\bullet\rightarrow\bullet$  : $n$ unmarked segments: Nimber is $0$ for $n$ even and $1$ for $n$ odd.

            \item $\bullet\rightarrow\bullet$---$\bullet\dots\bullet$---$\bullet\leftarrow\bullet$: $n$ unmarked segments: Nimber is $1$ for $n$ even and $0$ for $n$ odd.

            \item $\bullet\leftarrow\bullet$---$\bullet\dots\bullet$---$\bullet\rightarrow\bullet$  : $n$ unmarked segments: Nimber is $1$ for $n$ even and $0$ for $n$ odd.
        \end{enumerate}

    Consider the base case where $n=1$. The game is instantly winning in game type 4, $g_4$, because there is one possible move where player 1 wins. In types 2, 3, 5, and 6, the only possible move creates a source or sink, resulting in a Nimber of 0.
        
    Assume the inductive hypothesis holds for all $1\leq k \leq n$. To calculate the Nimbers of games of length $n+1$ we can split the games into pairs of sub-games, $\{g_p(a), g_q(b)\}$ of whose Nimbers we already know due to the inductive hypothesis since $a,b < n+1$. We can then use the xor function ($\oplus$) to combine the two games and find the resulting Nimber. We can split the game along every single possible segment and find the mex of these Nimbers to find the true Nimber value of the larger game. 
    
    For the base cases of $n \in \{2,3\}$, the induction hypothesis can be verified by exhaustive case analysis, which we omit here and leave to the reader to verify.
    
    We now establish the induction for all the cases for length $n+1$, where $n>2$, assuming that the induction hypothesis holds for all lengths up to and including $n$.

    \begin{enumerate}
        \item $\bullet$---$\bullet$---$\bullet\dots\bullet$---$\bullet$---$\bullet$ : $n+1$ segments
        \begin{enumerate}
        \item $n+1$ is even

        Because of the no source/sink rule, the edges at the ends cannot be marked.
        Suppose, then, that the edge that is marked splits the segment into two parts of length $a$ and $b$, where $a \leq b$, and $a+b=n$, where $n$ is odd. No matter how the edge is marked, this reduces to two sub-games, $g_3$ of the form 
        $\bullet\leftarrow\bullet$---$\bullet\dots\bullet$---$\bullet$---$\bullet$ and $g_2$ of the form $\bullet\rightarrow\bullet$---$\bullet\dots\bullet$---$\bullet$---$\bullet$.
        
        First we must consider the special case where $a=1$. In this case, one segment has Nimber $0$, while the other has Nimber $b-1=(n-1)-1=n-2$. Because $n>2$, this Nimber will always be greater than $0$. And of course, $(n-2) \oplus 0 = n-2$, which is greater than $0$.

        Next we consider the general case where $b \ge a >1$. In this case, by the induction hypothesis, the Nimbers of the sub-games will be $b-1$ and $a-1$. Now, because $n=a+b$ is odd, we know that exactly one of $a-1$ and $b-1$ must be odd. Therefore, $(a-1) \oplus (b-1)$ cannot be $0$.

        Thus we have shown that none of the sub-games of this case can have Nimber 0. Therefore, the mex of the Nimbers of all the sub-games must be 0, and this establishes the induction hypothesis for this case.

        \item $n+1$ is odd.
        \newline
        Similarly to the previous case, we can split the line into 2 sub-games with lengths of $a$ and $b$. 
        
        First we can show that the Nimber 1 will never appear through the xor of any pair of sub-games. In order to achieve a Nimber of 1 through xor, the two numbers must be odd and even. However, since $n+1$ is odd, $n$ must be even, so $a+b$ must add to an even number. An odd number added to an even number is odd, not even, therefore the Nimber 1 will never appear through xor.

        Next we can show that the Nimber 0 must appear. Since $a+b$ will be even, the case when $a=b$ will occur when split at the middle segment. These two sub-games are $g_3$ and $g_2$ of type $\bullet$---$\bullet$---$\bullet\dots\bullet$---$\bullet\rightarrow\bullet$ and $\bullet\rightarrow\bullet$---$\bullet\dots\bullet$---$\bullet$---$\bullet$ respectively, which both have the same Nimber, the xor of which is 0.

        Thus we have shown that none of the sub-games of this case can have value 1, and the value of 0 is present. Therefore, the mex of the values of all the sub-games must be 1, and this establishes the induction hypothesis for this case.
        \end{enumerate}

        \item $\bullet\rightarrow\bullet$---$\bullet\dots\bullet$---$\bullet$---$\bullet$ : $n+1$ segments.
        \newline
        Similarly to the previous case, we can split the line into 2 sub-games with the length of the left sub-game being $a$ and the length of the right sub-game being $b$(we sometimes will call the sub-games by their lengths $a$ and $b$). However, the left sub-game will either be of the form $\bullet\rightarrow\bullet$---$\bullet\dots\bullet$---$\bullet\rightarrow\bullet$ or $\bullet\rightarrow\bullet$---$\bullet\dots\bullet$---$\bullet\leftarrow\bullet$, with a Nimber of 0 or 1 depending on the parity of $a$, this pattern can be seen in the initial hypothesis. 
        
        First we observe the two end cases, where $a=0$, and $b=1$. In the case of $a=0$ the $b$ sub-game turns into the $n$ case, with a Nimber of $n-1$. In the case of $b=1$, the arrow can point in either direction. Depending on the parity of $a$ we can choose the direction such that the $a$ sub-game has a Nimber of 0. The $b$ sub-game has a Nimber of 0, and $0\oplus0$ is 0.

        Next we can show that the $b=1$ base case can be extended to include all the Nimbers up to the base case of $a=0$. By moving the partition of the sub-games, increasing $b$ and decreasing $a$, we can increase the Nimber of the $b$ sub-game, because it is of the form $\bullet\rightarrow\bullet$---$\bullet\dots\bullet$---$\bullet$---$\bullet$ which has a Nimber of $b-1$. By alternating the direction of the arrow according to the parity of $a$ we can ensure the $a$ sub-game has a Nimber of 0. And the xor of any number with 0 is just the number itself. So we can achieve all the Nimbers from 0 to $n-2$ this way.

        Finally we can show that any Nimbers higher than $n-1$ are impossible. The highest Nimber that can be achieved through either the $a$ or $b$ sub-games is n-2, excluding the $a=0$ case. This Nimber could be xored with at most a Nimber of 1, therefore only at most increasing the Nimber by 1 equaling $n-1$. Any other combination of sub-games must include either a 0 or 1 and therefore must be smaller.

        Thus we have shown that the sub-games of this case can have value 0 to $n-1$, and nothing higher than $n-1$. Therefore, the mex of the values of all the sub-games must be $n$, and this establishes the induction hypothesis for this case.

        \item $\bullet\leftarrow\bullet$---$\bullet\dots\bullet$---$\bullet$---$\bullet$ : $n+1$ segments.
        \newline
        This case is exactly the same as the previous one, except the arrow is the other way and as such the left sub-games will be of the form $\bullet\leftarrow\bullet$---$\bullet\dots\bullet$---$\bullet\rightarrow\bullet$ or $\bullet\leftarrow\bullet$---$\bullet\dots\bullet$---$\bullet\leftarrow\bullet$. However, the Nimbers are still the same and at least one of them will equal 0 for all $n$, the parity is simply flipped.

        \item $\bullet\rightarrow\bullet$---$\bullet\dots\bullet$---$\bullet\rightarrow\bullet$ : $n+1$ segments 
        \begin{enumerate}
        \item $n+1$ is even.
        \newline
        Similarly to the previous case, we can split the line into 2 sub-games with lengths of $a$ and $b$. 
        
        First we can show that the only Nimbers that can appear are 1s and 0s. Due to the two arrows at the ends, any pair of sub-games will be of the form $\bullet\rightarrow\bullet$---$\bullet\dots\bullet$---$\bullet\rightarrow\bullet$, $\bullet\leftarrow\bullet$---$\bullet\dots\bullet$---$\bullet\rightarrow\bullet$, or $\bullet\rightarrow\bullet$---$\bullet\dots\bullet$---$\bullet\leftarrow\bullet$, all of which can only have Nimbers of 0s and 1s.

        Next we can show that the only Nimber present is 1. Since $n+1$ is even, $n$ must be odd, as such the two sub-games, $a$ and $b$, must be of odd and even length, or vice versa. The two possibilities for types of sub-games are the odd version of $\bullet\rightarrow\bullet$---$\bullet\dots\bullet$---$\bullet\rightarrow\bullet$ with its even form, and the odd version of $\bullet\rightarrow\bullet$---$\bullet\dots\bullet$---$\bullet\leftarrow\bullet$ with the even form of $\bullet\leftarrow\bullet$---$\bullet\dots\bullet$---$\bullet\rightarrow\bullet$ and vice versa. This leads to the xor of 1 and 0 every single time, which is 1.

        Thus we have shown that none of the sub-games of this case can have value 0. Therefore, the mex of the values of all the sub-games must be 0, and this establishes the induction hypothesis for this case.

        \item $n+1$ is odd.
        \newline
        This is the almost the same case as before just with the parity flipped. This leads to the xor of 0 and 0, and 1 with 1, leading to only 0 being a Nimber. Therefore the mex of the values of all the sub-games must be 1.
        \end{enumerate}
        \item $\bullet\rightarrow\bullet$---$\bullet\dots\bullet$---$\bullet\leftarrow\bullet$ and $\bullet\leftarrow\bullet$---$\bullet\dots\bullet$---$\bullet\rightarrow\bullet$: $n+1$ segments.
        \newline
        We two types 5 and 6 simultaneously since the have the same Nimbers. Type 5 games turn into types 4 and 5, and type 6 games turn into types 4 and 6. They are shown to have Nimber 0 in a similar manner to case 4, with parity flipped.
    \end{enumerate}

\end{proof}

This theorem leads nicely into the corresponding case for a graph consisting of a single cycle. However, we omit this because, after turning our attention to the case of the Cycles game with no sources allowed, we found a more general result that can be used to cover all of these cases.

\section{Cycles with Sources on a Line}

\begin{theorem}
    Let $\{g_1(n), g_2(n), \dots, g_k(n)\}$ be a set of impartial games for all natural numbers $n$. For every $i\in [k]$, there exists a subset $S_i$ of $[k]\times[k]$. Suppose for all $n$, every playable move, denoted $(p,q,a,b)$ on $g_i(n)$ splits the game into $g_p(a)$ and $g_q(b)$, where $(p,q)\in S_i$, $a$ and $b$ are non-negative integers, and $a+b+1=n$. Let $s\geq T$ be positive integers. Further, suppose:
    \begin{enumerate}
        \item For all $a,b\geq s-T$ and $(p,q)\in S_i$, the move $(p,q,a,b)$ is valid for $n=a+b+1$.
        \item If $a< s-T$, and $n-a-1\geq s-T$, then if $(p,q,a,n-a-1)$ is invalid for some $g_i(n)$, then $(p,q,a,m-a-1)$ is invalid for all natural numbers $m$,
        \item Similarly, if $b<s-T$, and $n-b-1\geq s-T$, then if $(p,g,n-b-1,b)$ is invalid for some $g_i(n)$, then $(p,g,m-b-1,b)$ is invalid for all natural numbers $m$.
    \end{enumerate}  
    
    Finally, suppose also that for all $n \in [s+1,2s+1]$, we have that $\forall i\in [k]$, the Nimber of $g_i(n)$ is equal to the Nimber of $g_i(n-T)$, which we can write as $Nim(g_i(n)) = Nim(g_i(n-T))$. 
    
    Then, $\forall n \geq s+1$, $Nim(g_i(n))=Nim(g_i(n-T))$ for all $i\in [k]$.

\end{theorem}
\begin{proof}
    We prove by induction. The base case, for $s+1\leq n \leq 2s+1$, is given.
    \newline\newline
    We suppose $n>2s+1$ and prove the inductive hypothesis for any $i=1,\dots,k$, that $Nim(g_i(n))=Nim(g_i(n-T))$. Suppose that the hypothesis holds for all $g_i(m)$ where $s+1\leq m < n$.
    \newline\newline
    Consider any valid move $(p,q,a,b)$ on a fixed $g_i(n)$, where $(p,q) \in S_i$ and $a+b+1=n$. We will consider pairs of Nimbers formed where at least one of $a$ and $b$ is less than or equal to $s$, which we will call the ``outer $2s$'' cases. 
    We first suppose $a\leq s$. To avoid violating condition 2, moves of the form $(p,q,a,b_0)$ must be valid for all $b_0\geq s-T$. In particular, $b-T\geq s-T$, so the game $g_i(n-T)$ can be split into $g_p(a)$ and $g_q(b-T)$. Since $n > b \geq s+1$, the inductive hypothesis applies to $b$, meaning $Nim(g_q(b)) = Nim(g_q(b-T))$. We therefore have $Nim(g_p(a)) \oplus Nim(g_q(b)) = Nim(g_p(a)) \oplus Nim(g_q(b-T))$. This same argument applies symmetrically to the case where $b \leq s$, evoking condition 3. This means that the Nimbers of   the ``outer $2s$'' pairs in $g_i(n)$ and their xors are the same as those of the ``outer $2s$'' of $g_i(n-T)$. Note that the ``outer $2s$'' of $g_i(n-T)$ may not be $2s$ distinct moves and may overlap if $n-T \leq 2s$.
    \newline\newline
    Now we show that the Nimbers of the ``middle'' pairs, the pairs for which $a,b>s$ (which are always playable for any $(p,q)\in S_i$ by condition 1), are the same as the pairs of Nimbers of the ``outer $2s$.'' If $b>s$, we can repeatedly invoke the inductive hypothesis until we reach a move $(p,q,a-jT,b+jT)$, which is valid by condition 1. $Nim(g_p(a-jT))\oplus Nim(g_q(b+jT)) = Nim(g_p(a))\oplus Nim(g_q(b))$. Therefore, the ``middle'' pairs have duplicate Nimbers of the valid ``outer $2s$'' pairs.
    \newline\newline
    We have proved that every pair of Nimbers formed by playing a valid move in $g_i(n)$ is one of the pairs of Nimbers formed from an ``outer $2s$'' move of $g_i(n)$, which  in turn is a pair of Nimbers formed from an ``outer $2s$'' move of $g_i(n-T)$. In addition, the Nimbers of every ``middle'' pair in $g_i(n-T)$ are also found in the pairs of the ``outer $2s$'' of $g_i(n-T)$. Therefore, every pair of Nimbers in either game is found in the Nimbers of the shared ``outer $2s$'' pairs, and, trivially, every Nimber pair in the shared ``outer $2s$'' is in both games. Therefore, $g_i(n)$ and $g_i(n-T)$ contain all of the same pairs of Nimbers and their xors, and so, taking the mex, $Nim(g_i(n)) = Nim(g_i(n-T))$.
\end{proof}
The motivation for this Theorem was a pattern we noticed in the Nimbers of the Game of Cycles on a line of length $n$ with sources allowed. In particular, the ``outer $2s$'' above was in fact an ``outer 86'' in the various game types that emerge when examining this game.
 \begin{theorem}
    For Cycles with sources the Nimbers of line segments of successive length are given by the following sequence:
    The first 18 are:
    \begin{equation*}
        0,1,0,1,0,3,2,0,2,3,0,1,0,1,0,5,7,0
    \end{equation*}
    From then on, the Nimbers form the following repeated sequence
    \begin{equation*}
        1,0,1,0,3,2,4,5,3,0,1,0,1,0,5,7,8
    \end{equation*}
    \end{theorem}

    \begin{proof}
        Let $\{g_1(n),\dots,g_6(n)\}$ be the following types of games, where $n$ is the number of unmarked edges.
        \begin{enumerate}
            \item $\bullet$---$\bullet\dots\bullet$---$\bullet$
    
        \item $\bullet\rightarrow\bullet$---$\bullet\dots\bullet$---$\bullet$

        \item $\bullet\leftarrow\bullet$---$\bullet\dots\bullet$---$\bullet$

        \item $\bullet\rightarrow\bullet$---$\bullet\dots\bullet$---$\bullet\rightarrow\bullet$

        \item $\bullet\rightarrow\bullet$---$\bullet\dots\bullet$---$\bullet\leftarrow\bullet$

        \item $\bullet\leftarrow\bullet$---$\bullet\dots\bullet$---$\bullet\rightarrow\bullet$
        \end{enumerate}
        For each game type $g_i$, moves are always of the form $(p,q,a,b)$, where $(p,q)=(2,3)$ or $(3,2)$ for $i=1$, $(p,q)=(5,3)$ or $(4,2)$ for $i=2$, $(p,q)=(4,3)$ or $(6,2)$ for $i=3$, $(p,q)=(4,4)$ or $(5,6)$ for $i=4$, $(p,q)=(4,5)$ or $(5,4)$ for $i=5$, and $(p,q)=(4,6)$ or $(6,4)$ for $i=6$.
        The only unplayable moves are moves where $a$ or $b$ is zero, so condition 1 of Theorem 2 is satisfied. 
        The unplayable moves are:
        \begin{itemize}
            \item Placing a left arrow on the left edge of type 2, denoted $(5,3,0,n-1)$ for $g_2(n)$, $\forall n$
            \item Placing a left arrow on the left edge of type 4, denoted $(5,6,0,n-1)$ for $g_4(n)$, $\forall n$
            \item Placing a left arrow on the left edge of type 5, denoted $(5,4,0,n-1)$ for $g_5(n)$, $\forall n$
            \item Placing a right arrow on the right edge of type 5, denoted $(4,5,n-1,0)$ for $g_5(n)$, $\forall n$
        \end{itemize}
        Because any unplayable move is of this form, conditions 2 and 3 hold. Lastly, a computer program can verify the first 87 Nimbers of all 6 game types:
         \begin{enumerate}
        \item $\bullet$---$\bullet\dots\bullet$---$\bullet$: \newline$0,1,0,1,0,3,2,0,2,3,0,1,0,1,0,5,7,0,1,0,1,0,3,2,4,5,3,0,1,0,1,0,5,7,8,1,0,1,0,3,2,4,5,3,0,1,0,1,0,$\newline$5,7,8,1,0,1,0,3,2,4,5,3,0,1,0,1,0,5,7,8,1,0,1,0,3,2,4,5,3,0,1,0,1,0,5,7,8,1$
    
        \item $\bullet\rightarrow\bullet$---$\bullet\dots\bullet$---$\bullet$  :  \newline$0,1,0,1,0,3,2,0,2,3,0,1,0,1,0,5,7,0,1,0,1,0,3,2,4,5,3,0,1,0,1,0,5,7,8,1,0,1,0,3,2,4,5,3,0,1,0,1,0,$\newline$5,7,8,1,0,1,0,3,2,4,5,3,0,1,0,1,0,5,7,8,1,0,1,0,3,2,4,5,3,0,1,0,1,0,5,7,8,1$

        \item $\bullet\leftarrow\bullet$---$\bullet\dots\bullet$---$\bullet$  :\newline$1,2,3,1,3,2,4,5,2,3,1,3,2,1,4,2,4,1,2,3,1,3,2,4,5,2,3,1,3,2,1,4,3,4,1,2,3,1,3,2,4,5,9,3,1,3,2,1,4,$\newline$3,4,1,2,3,1,3,2,4,5,9,3,1,3,2,1,4,3,4,1,2,3,1,3,2,4,5,9,3,1,3,2,1,4,3,4,1,2$

        \item $\bullet\rightarrow\bullet$---$\bullet\dots\bullet$---$\bullet\rightarrow\bullet$  :\newline$1,2,3,1,3,2,4,5,2,3,1,3,2,1,4,2,4,1,2,3,1,3,2,4,5,2,3,1,3,2,1,4,3,4,1,2,3,1,3,2,4,5,9,3,1,3,2,1,4$\newline$3,4,1,2,3,1,3,2,4,5,9,3,1,3,2,1,4,3,4,1,2,3,1,3,2,4,5,9,3,1,3,2,1,4,3,4,1,2$

        \item $\bullet\rightarrow\bullet$---$\bullet\dots\bullet$---$\bullet\leftarrow\bullet$:\newline$0,1,0,1,0,3,2,0,2,3,0,1,0,1,0,5,7,0,1,0,1,0,3,2,4,5,3,0,1,0,1,0,5,7,8,1,0,1,0,3,2,4,5,3,0,1,0,1,0,$\newline$5,7,8,1,0,1,0,3,2,4,5,3,0,1,0,1,0,5,7,8,1,0,1,0,3,2,4,5,3,0,1,0,1,0,5,7,8,1$

        \item $\bullet\leftarrow\bullet$---$\bullet\dots\bullet$---$\bullet\rightarrow\bullet$  : \newline$1,0,3,2,0,2,3,0,1,0,1,0,5,7,0,1,0,1,0,3,2,4,5,3,0,1,0,1,0,5,7,8,1,0,1,0,3,2,4,5,3,0,1,0,1,0,$\newline$5,7,8,1,0,1,0,3,2,4,5,3,0,1,0,1,0,5,7,8,1,0,1,0,3,2,4,5,3,0,1,0,1,0,5,7,8,1,0$
    \end{enumerate}
    Evidently, $\forall n \in [44,87]$, for each $i \in [6]$, $Nim(g_i(n))=Nim(g_i(n-17))$. Therefore, by the theorem, for each $i \in [6]$, $Nim(g_i(n))=Nim(g_i(n-17))$ $\forall n \geq 88$. Therefore, the proposed repetitive sequence holds.
    \end{proof}

    \section{Graphs Consisting of a Single Cycle}

    \ctikzfig{fig1}
    
    \begin{figure}[h]
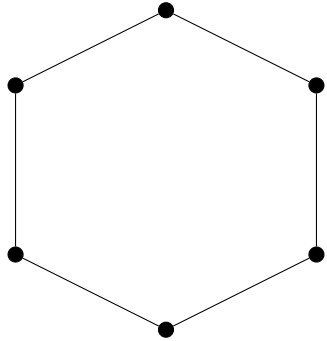

    \caption{A simple cycle of length 6}
    \label{fig1}
    \end{figure}
    \begin{theorem}
        With the standard ruleset, a graph consisting of a single cycle will have Nimber 0 if the amount of edges, $n$, is even and will have Nimber 1 if $n$ is odd for all $n>1$.
    \end{theorem}

    \begin{proof}
        A cycle game, upon marking an edge, can be viewed as a line, with the first marked edge being both the starting and ending edge. In order to accurately represent a cycle as a line the nodes at the start and end will not be constricted by the no source or sink rule. In this manner the game can be viewed to consist of 3 game types denoted as the set $\{g_1(n),g_2(n),g_3(n)\}$ where $n$ is the number of unmarked edges.

        \begin{enumerate}
            \item $\bullet\rightarrow\bullet$---$\bullet\dots\bullet$---$\bullet\rightarrow\bullet$
            \item $\bullet\rightarrow\bullet$---$\bullet\dots\bullet$---$\bullet\leftarrow\bullet$
            \item $\bullet\leftarrow\bullet$---$\bullet\dots\bullet$---$\bullet\rightarrow\bullet$
        \end{enumerate}

        For each game type, $g_i(n)$, a move must be of the form $(p,q,a,b)$, where if $i=1$ then $(p,q)=(1,1)$ or $(p,q)=(2,3)$, if $i=2$ then $(p,q)=(1,2)$ or $(p,q)=(2,1)$, if $i=3$ then $(p,q)=(1,3)$ or $(p,q)=(3,1)$. Unplayable moves only exist if $a$ or $b$ are 0, and the original cycle game of length $n=1$ does not exist, satisfying condition 1 of Theorem 2. A list of the unplayable moves is given below:

        \begin{itemize}
            \item On $g_1(n)$ marking the leftmost edge to face the left node, denoted $(2,3,0,n-1)$ for $n>1$.
            \item On $g_1(n)$ marking the rightmost edge to face the left node, denoted $(2,3,n-1,0)$ for $n>1$.
            \item On $g_2(n)$ marking the leftmost edge to face the left node, denoted $(2,1,0,n-1)$ for $n>1$.
            \item On $g_2(n)$ marking the rightmost edge to face the right node, denoted $(3,1,n-1,0)$ for $n>1$.
            \item On $g_3(n)$ marking the leftmost edge to face the right node, denoted $(3,1,0,n-1)$ for $n>1$.
            \item On $g_3(n)$ marking the rightmost edge to face the left node, denoted $(1,3,n-1,0)$ for $n>1$.
        \end{itemize}

        These forms of unplayable moves exist for all $n\geq1$, satisfying conditions 2 and 3. The initial case exists when $s=3$ and the period of the pattern is two, $T=2$, following the parity of $n$. The Nimbers of $n=1\dots7$ can be verified by computer.

        \begin{enumerate}
            \item $\bullet\rightarrow\bullet$---$\bullet\dots\bullet$---$\bullet\rightarrow\bullet$  : $n=1\dots7$ unmarked segments: Nimber is $1,0,1,0,1,0,1$ respectively.

            \item $\bullet\rightarrow\bullet$---$\bullet\dots\bullet$---$\bullet\leftarrow\bullet$: $n=1\dots7$ unmarked segments: Nimber is $0,1,0,1,0,1,0$ respectively.

            \item $\bullet\leftarrow\bullet$---$\bullet\dots\bullet$---$\bullet\rightarrow\bullet$  : $n=1\dots7$ unmarked segments: Nimber is $0,1,0,1,0,1,0$ respectively.
        \end{enumerate}

        Clearly for all $g_i(n)$, where $i=1,2,3$ and $n\in[4,7]$, $Nim(g_i(n))=Nim(g_i(n-2))$. This allows the game to satisfy all requirements of Theorem 2 allowing the pattern of its sub-games to continue on for all $n>7$. However, it should be noted that the first player always turns the game into $g_1(n-1)$, switching the Nimbers around for the Nimber of the fully unmarked game from that of $g_1(n-1)$. For example, if $n=5$, the first player turns the unmarked game into $g_1(4)$, with Nimber 0, which suggests a second player victory, but now, the first player \textit{is} the second player in this state, showing that the first player is in the winning position, therefore the original unmarked game must have a Nimber of 1.

        \begin{itemize}
            \item For $n$ even:
            \begin{enumerate}
                \item $\bigcirc$: $n$ unmarked segments: Nimber is 0
                
                \item $\bullet\rightarrow\bullet$---$\bullet\dots\bullet$---$\bullet\rightarrow\bullet$  : $n$ unmarked segments: Nimber is $0$.
    
                \item $\bullet\rightarrow\bullet$---$\bullet\dots\bullet$---$\bullet\leftarrow\bullet$: $n$ unmarked segments: Nimber is $1$.
    
                \item $\bullet\leftarrow\bullet$---$\bullet\dots\bullet$---$\bullet\rightarrow\bullet$  : $n$ unmarked segments: Nimber is $1$.
            \end{enumerate}
            \item For $n$ odd:
            \begin{enumerate}
                \item $\bigcirc$: $n$ unmarked segments: Nimber is 1
                
                \item $\bullet\rightarrow\bullet$---$\bullet\dots\bullet$---$\bullet\rightarrow\bullet$  : $n$ unmarked segments: Nimber is $1$.
    
                \item $\bullet\rightarrow\bullet$---$\bullet\dots\bullet$---$\bullet\leftarrow\bullet$: $n$ unmarked segments: Nimber is $0$.
    
                \item $\bullet\leftarrow\bullet$---$\bullet\dots\bullet$---$\bullet\rightarrow\bullet$  : $n$ unmarked segments: Nimber is $0$.
            \end{enumerate}
        \end{itemize}
    \end{proof}
    Amazingly, despite how complicated the game of cycles becomes when dealing with allowing sources on a line, allowing sources on a cycle produces a simple result.
     \begin{theorem}
        For Cycles with Sources every simple cycle has Nimber 0.
    \end{theorem}
    \begin{proof}
        Just like the previous game, this can also be viewed with sub-games of lines. When marking an edge for the first time, the game is reduced to a game type 4, $g_4(n)$, in the Theorem 2 proof. This game type, as proved in the first theorem, does not at any point have a Nimber of 0. Since getting to this position takes a move, just like the previous proof, the Nimbers are swapped as the first player cannot reduce the unmarked game to a Nimber of 0, resulting in the unmarked game having a Nimber of 0 for all $n$.

    \end{proof}
    Player 1 is always losing when playing on a cycle with sources allowed, but the winning strategy is not as immediately obvious as winning strategies for cases of the game when sources are disallowed.

    \section*{Acknowledgements}
    We would like to thank Professor Amit Sahai for introducing us to ``The Game of Cycles'' and giving us direction on what to research, specifically the idea of changing the rules of the game, and for giving advice on how to write a research paper. We further thank Professor Sahai for helping us learn Combinatorial Game Theory and how to use Latex. We also thank Owen Maitzen and Gaurav Sen for creating amazing YouTube videos explaining Sprague–Grundy Theory~\cite{hackenbush,grundy}.

    \bibliographystyle{plain}
    \bibliography{biblo.bib}

\footnotesize
    \definecolor{codegreen}{rgb}{0,0.6,0}
    \definecolor{codegray}{rgb}{0.5,0.5,0.5}
    \definecolor{codepurple}{rgb}{0.58,0,0.82}
    \definecolor{backcolour}{rgb}{0.95,0.95,0.92}
    
    \lstdefinestyle{mystyle}{
        backgroundcolor=\color{backcolour},   
        commentstyle=\color{codegreen},
        keywordstyle=\color{magenta},
        numberstyle=\tiny\color{codegray},
        stringstyle=\color{codepurple},
        basicstyle=\ttfamily\footnotesize,
        breakatwhitespace=false,         
        breaklines=true,                 
        captionpos=b,                    
        keepspaces=true,                 
        numbers=left,                    
        numbersep=5pt,                  
        showspaces=false,                
        showstringspaces=false,
        showtabs=false,                  
        tabsize=2
    }
    
    \lstset{style=mystyle}
    
    \section{Code Appendix}
    \begin{lstlisting}[language=C++]
#include <string>
#include <vector>
#include <map>
#include <iostream>
#include <algorithm>
using namespace std;

/*
Cycles played on a line, yes sources but no sinks

To compute Nimbers this code will use a Dynamic Programming approach, using previously
collected subgames to eventually mex together and find the Nimber of the current game.
It is recomended to be familiar with Dynamic Programming to understand this program.
*/

//Function for calculating the minimum excluded value, or mex.
int mex(vector<int> v){
    if(v.size()==0) return 0;
    sort(v.begin(),v.end());
    if(v[0]!=0) return 0;
    for(int i=1;i<v.size();i++){
        if(v[i-1]+1!=v[i] && v[i-1]!=v[i]) return v[i-1]+1;
    }
    return v[v.size()-1]+1;
}

int main(){

    cout << "Please enter the largest number of unmarked edges you wish to observe the Nimbers for:" << endl;
    //'n' being the amount of unmarked segments.
    int n;
    cin >> n;
    /*
    LineMap is a vector of sets used for storing each unqiue subgame for Dynamic
    Programing, where the amount of unmarked edges, n, dictates in which map the
    subgames will be stored.
    */
    vector<map<string, int>> LineMap(n);
    /*
    The base case for n=1 is stored in the first map, LineMap[0], each subgame
    will be stored with its Nimber.

    Notice that the mirror of each subgame is present, this is not shown in the
    proof as they are equivalent, however, for simplicity's sake the code will 
    analyze every case regardless of equivalency through mirroring.
    */
    LineMap[0]["-"] = 0;
    LineMap[0][">-"] = 0;
    LineMap[0]["<-"] = 1;
    LineMap[0]["-<"] = 0;
    LineMap[0]["->"] = 1;
    LineMap[0][">-<"] = 0;
    LineMap[0]["<->"] = 1;
    LineMap[0][">->"] = 1;
    LineMap[0]["<-<"] = 1;

    /*
    This for loop will loop through all subgames of length [2,n], where "lineMapIndex"
    is the index variable dictating where in the LineMap the subgames will be stored,
    the length of the current game being computed is lineMapIndex+1, in the commentary
    this value will be referred to as curN.
    */
    for(int lineMapIndex=1;lineMapIndex<n;lineMapIndex++){
        /*
        "line" represents the game of length curN, it starts off with all of its
        segments as unmarked which is notated as '-'.
        */
        string line = "";
        for(int i=0;i<=lineMapIndex;i++){
            line = line + "-";
        }
        /*
        The vector "cases" will store the Nimbers of each subgame's subgame pairs
        inorder to use the mex function to calculate the Nimber of the subgame, each
        subgame is represented with a number as shown below:
        
        Note that case 0 is the orginial unmarked game of length curN.
        */
        vector<vector<int>> cases(9);
        /*
        The Cases represent the different subgame types depending on how the edges of
        the two ends are marked. REMEBER THAT UNMARKED EDGES AT THE END MUST OBEY THE
        "NO SINK" RULE, however, the already marked edges do not need to obey the
        rule because we assume that they are "attatched" to other games and thus their
        marking did previously obey the "no sink" rule.
        Case 0: -...-
        Case 1: >-...-
        Case 2: <-...-
        Case 3: -...->
        Case 4: -...-<
        Case 5: >-...->
        Case 6: <-...-<
        Case 7: <-...->
        Case 8: >-...-<
        

        
        These are edge cases for when the first edge is marked, creating subgames of
        a=0 and b=curN-1. The according Nimber is found through Dynamic Programing
        as the values of any subgames of length less than curN have already been
        computed.

        Certain cases are absent because those cases are not possible due to the
        restriction that moves are not allowed to create sinks.
        */
        line[0] = '<';
        /*
        Observe that cases 0,1,3,4,5,8 are absent due to the creation of a leftmost
        sink.
        */
        cases[2].push_back(LineMap[lineMapIndex-1][line]);
        cases[6].push_back(LineMap[lineMapIndex-1][line+"<"]);
        cases[7].push_back(LineMap[lineMapIndex-1][line+">"]);
        line[0] = '>';
        cases[0].push_back(LineMap[lineMapIndex-1][line]);
        cases[1].push_back(LineMap[lineMapIndex-1][line]);
        cases[2].push_back(LineMap[lineMapIndex-1][line]);
        cases[3].push_back(LineMap[lineMapIndex-1][line+">"]);
        cases[4].push_back(LineMap[lineMapIndex-1][line+"<"]);
        cases[5].push_back(LineMap[lineMapIndex-1][line+">"]);
        cases[6].push_back(LineMap[lineMapIndex-1][line+"<"]);
        cases[7].push_back(LineMap[lineMapIndex-1][line+">"]);
        cases[8].push_back(LineMap[lineMapIndex-1][line+"<"]);
        line[0] = '-';
        /*
        This for loop will loop through all pairs of subgames except when a=0
        (the above cases) or b=0.
        */
        for(int i=1;i<lineMapIndex;i++){
            /*
            The varibles "nl...rr" represent the Nimber for each subgame's two
            subgames, when marked at position i, resulting in a LEFT subgame and a
            RIGHT subgame. The variable name consists of two letters, a first letter
            and a second letter. The first letter denotes the marking of the edge at
            the end, not posititon i. 'n' stands for UNMARKED, 'l' stands for <, 'r'
            stands for >. The second letter of the variable name denotes if it is the
            LEFT or RIGHT subgame.
            
            For example,
            "nl" represents the subgame's LEFT subgame where the leftmost edge is
            UNMARKED.
            Of form: --...-(< or > at pos i)
            "rl" represents the subgame's LEFT subgame where the leftmost edge is
            marked to the RIGHT
            Of form: >-...-(< or > at pos i)
            "lr" represents the subgame's RIGHT subgame where the rightmost edge is
            marked to the LEFT
            Of form: (< or > at pos i)-...-<

            These are combined to create each case, for example case 5 is
            made up of "rl" and "rr":
            Of form: >-...-(< or >)-...->
            */
            int nl,rl,ll,nr,lr,rr = 0;
            //This is the direction of the marked edge at position i
            line[i] = '<';
            string leftLine = line.substr(0,i+1);
            string rightLine = line.substr(i, line.size()-i);
            int leftIndex = i;
            int rightIndex = line.size()-1-i;
            
            nl = LineMap[leftIndex-1][leftLine];
            rl = LineMap[leftIndex-1][">" + leftLine];
            ll = LineMap[leftIndex-1]["<" + leftLine];

            nr = LineMap[rightIndex-1][rightLine];
            lr = LineMap[rightIndex-1][rightLine + "<"];
            rr = LineMap[rightIndex-1][rightLine + ">"];

            /*
            Each case's Nimber is calcuated by using xor to combine the subgames'
            Nimbers.
            */
            cases[0].push_back(nl^nr);
            cases[1].push_back(rl^nr);
            cases[2].push_back(ll^nr);
            cases[3].push_back(nl^rr);
            cases[4].push_back(nl^lr);
            cases[5].push_back(rl^rr);
            cases[6].push_back(ll^lr);
            cases[7].push_back(ll^rr);
            cases[8].push_back(rl^lr);

            //This is the other direction for the edge marked at position i.
            leftLine[leftLine.size()-1] = '>';
            rightLine[0] = '>';

            nl = LineMap[leftIndex-1][leftLine];
            rl = LineMap[leftIndex-1][">" + leftLine];
            ll = LineMap[leftIndex-1]["<" + leftLine];
            
            nr = LineMap[rightIndex-1][rightLine];
            lr = LineMap[rightIndex-1][rightLine + "<"];
            rr = LineMap[rightIndex-1][rightLine + ">"];

            /*
            Each case's Nimber is calcuated by using xor to combine the subgames'
            Nimbers.
            */
            cases[0].push_back(nl^nr);
            cases[1].push_back(rl^nr);
            cases[2].push_back(ll^nr);
            cases[3].push_back(nl^rr);
            cases[4].push_back(nl^lr);
            cases[5].push_back(rl^rr);
            cases[6].push_back(ll^lr);
            cases[7].push_back(ll^rr);
            cases[8].push_back(rl^lr);
            //Unmarking the edge at position i for the next iteration.
            line[i] = '-';
        }
        /*
        These are edge cases for when the last segment is marked creating subgames of
        a=curN-1 and b=0, the same logic for when a=0 appiles here.
        */
        line[lineMapIndex] = '<';
        cases[0].push_back(LineMap[lineMapIndex-1][line]);
        cases[1].push_back(LineMap[lineMapIndex-1][">"+line]);
        cases[2].push_back(LineMap[lineMapIndex-1]["<"+line]);
        cases[3].push_back(LineMap[lineMapIndex-1][line]);
        cases[4].push_back(LineMap[lineMapIndex-1][line]);
        cases[5].push_back(LineMap[lineMapIndex-1][">"+line]);
        cases[6].push_back(LineMap[lineMapIndex-1]["<"+line]);
        cases[7].push_back(LineMap[lineMapIndex-1]["<"+line]);
        cases[8].push_back(LineMap[lineMapIndex-1][">"+line]);
        line[lineMapIndex] = '>';
        /*
        Observe that cases 0,1,2,4,6,8 are absent due to the creation of a rightmost
        sink
        */
        cases[3].push_back(LineMap[lineMapIndex-1][line]);
        cases[5].push_back(LineMap[lineMapIndex-1][">"+line]);
        cases[7].push_back(LineMap[lineMapIndex-1]["<"+line]);

        line[lineMapIndex] = '-';
        /*
        Now we calculate the mex of all the Nimbers of the subgames to find the
        Nimber of each case
        */
        LineMap[lineMapIndex][line] = mex(cases[0]);
        LineMap[lineMapIndex][">" +line] = mex(cases[1]);
        LineMap[lineMapIndex]["<" +line] = mex(cases[2]);
        LineMap[lineMapIndex][line+ ">"] = mex(cases[3]);
        LineMap[lineMapIndex][line+ "<"] = mex(cases[4]);
        LineMap[lineMapIndex][">" +line+ ">"] = mex(cases[5]);
        LineMap[lineMapIndex]["<" +line+ "<"] = mex(cases[6]);
        LineMap[lineMapIndex]["<" +line+ ">"] = mex(cases[7]);
        LineMap[lineMapIndex][">" +line+ "<"] = mex(cases[8]);

    }
    cout << endl << endl;
    cout << "All \',\'s are 0s, and all \'.\'s are 1s" << endl;
    cout << "Each column i represents the Nimber for the subgames of length i, each game type is shown to the \nright of each row. The type is determined by the marking of the end segments which are shown:" << endl;
    //Printing the results in one massive line to see patterns.
    cout << endl;
    int aa=0;
    for(auto z : LineMap[0]){
        for(int b=0;b<n;b++){
            int c = 0;
            for(auto i : LineMap[b]){
                if(c==aa){
                    /*
                    In order to more easily see the pattern, 1s and 0s are turned
                    into ','s and '.'s respectively.
                    */
                    if(i.second==0){
                        cout << ",";
                    }
                    else if(i.second==1){
                        cout << ".";
                    }
                    else cout << i.second;
                }
                c++;
            }
        }
        cout << "  :"<<z.first;
        cout << endl;
        aa++;
    }

    //Printing the results of each case to analyze separately if desired
    cout << endl;
    for(int a=0;a<n;a++){
        cout << "Length: " << a+1 << endl;
        for(auto i : LineMap[a]){
            cout << i.first << " = " << i.second << endl;
        }
    }

}
\end{lstlisting}
    
\end{document}